\documentclass[12pt]{amsart}
\usepackage{amsfonts,amssymb,amscd,amsmath,enumerate,verbatim,calc}
\usepackage{graphicx}

\newtheorem{Theorem}{Theorem}[section]
\newtheorem{Lemma}[Theorem]{Lemma}

\newtheorem{Proposition}[Theorem]{Proposition}
\newtheorem{Remark}[Theorem]{Remark}
\newtheorem{Remarks}[Theorem]{Remarks}
\newtheorem{Example}[Theorem]{Example}

\newtheorem{Definition}[Theorem]{Definition}

\begin{document}
\title{Resolving Share and Topological Index}
\author{Muhammad Salman, Imran Javaid$^*$, Muhammad Anwar Chaudhry}
\keywords{resolving set, resolving share, average resolving share, resolving topological index
\indent Supported by the Higher Education Commission of Pakistan (Grant No. 17-5-3(Ps3-257)\\ \indent HEC/Sch/2006)\\
\indent $^*$ Corresponding author: ijavaidbzu@gmail.com}
\address{Center for Advanced Studies in Pure and Applied Mathematics,
Bahauddin Zakariya University Multan, Pakistan\\
\newline \indent Email: solo33@gmail.com, ijavaidbzu@gmail.com, chaudhry@bzu.edu.pk}

\date{}
\maketitle
\begin{abstract}
An atom $a$ of a molecular graph $G$ uniquely determines (resolves)
a pair $(a_1,a_2)$ of atoms of $G$ if the distance between $a$ and
$a_1$ is different from the distance between $a$ and $a_2$. In this
paper, we quantify the involvement of each atom $a$ of $G$ in
uniquely determining (resolving) a pair $(a_1,a_2)$ of atoms of $G$,
which is called the resolving share of $a$ for the pair
$(a_1,a_2)$. Using this quantity, we define a distance-based
topological index of a molecular graph, which reflects the topology
of that molecular graph according to the resolvability behavior of
each of its atom, and is called the resolving topological index.
Then we compute the resolving topological index of several molecular
graphs.
\end{abstract}

\section{Introduction}
A major part of the current research in mathematical chemistry,
chemical graph theory and quantative structure-activity-property
relationship QSAR/QSPR studies involves topological indices.
Topological indices are numerical identities derived in an
unambiguous manner from a molecular graph \cite{17,23}. These
indices are graph invariants which usually characterize the topology
of that molecular graph. Some major classes of topological indices
such as distance-based topological indices, connectivity topological
indices and counting related polynomials and indices of graphs have
found remarkable employment in several chemistry fields.

The first non-trivial distance-based topological index was Wiener
index, introduced by Wiener in 1947 \cite{25}. To explain various
chemical and physical properties of atoms, molecules, and to
correlate the structure of molecules to their biological activity,
Wiener index plays a significant role \cite{12}. Caused by this
usefulness of the Wiener index, the research interest in Wiener
index and related distance-based indices is still considerable. In
the last twenty years, surprisingly a large number of modifications
and extensions of the Wiener index such as Schultz index $MTI(G)$,
proposed by Schultz \cite{20}; Szeged index $S_z(G)$, proposed by
Gutman \cite{4}; revised Wiener or revised Szeged index $S_{z^*}(G)$
proposed by Randi\'{c} \cite{18}; modified Wiener index for trees
$^mW(T)$, proposed by Nikoli\'{c} {\em et al.} \cite{13}; another
class of modified Wiener indices $^mW_{\lambda}(T)$, proposed by
Gutman {\em et al.} \cite{6}; Harary index $H(G)$, proposed by
Plavsi\'{c} {\em et al.} \cite{29} and Baladan index $J(G)$,
proposed by Baladan \cite{33} and by Randi\'{c} \cite{32}, to name a
few, was put forward and studied. An extensive bibliography on this
matter can be found in the reviews \cite{3,10}.

The problems on distance in graphs continues to seek the attention
of scientists both as theory and applications. Among these problems,
the most famous problem in graphs, which plays a vital role to
uniquely distinguish all the vertices of a graph, is resolvability.
Roughly speaking, by resolvability in a graph $G$, we mean that any
two vertices of $G$ in the pair $(u,v)$ are said to be uniquely
distinguished (represented or resolved) by a vertex $w$ of $G$ if
the distance between $u$ and $w$ is distinct from the distance
between $v$ and $w$. The vertices of $G$ needed to distinguish all
the pairs of the vertices of $G$ form a remarkable set known as a
resolving set for $G$, and it has a significant application in
pharmaceutical chemistry.

A fundamental problem in pharmaceutical chemistry is to find out the
unique representations of chemical compounds in several molecular
structures (graphs). The intention behind uniquely representing the
chemical compounds is to determine whether any two compounds in the
structure share the same functional group at a particular position.
This comparative statement plays a fundamental part in drug
discovery whenever it is to be determined whether the features of a
chemical compound are responsible for its pharmacological activity
\cite{14,15}. The solution of this fundamental problem was addressed
by the concept of resolving set. A minimum resolving set is, in
fact, the set of those few atoms in a molecular graph which
determine the unique representations of the chemical compounds. Now,
a question with remarkable interest arises that {\it how much an atom
of a molecular graph partake in uniquely representing any pair of
chemical compounds?}. Precisely, in a graph $G$, how much a vertex
of $G$ involves itself to resolve any pair of vertices of $G$? In
this paper, we fix this problem by defining the amount of
resolving done by an atom (vertex) to represent (resolve) every
pair of compounds of a molecular graph, and is called the resolving
share of that atom. With the help of resolving share of each atom, a
numeric identity is associated with the molecular graph, called the
resolving topological index, which reflects the total amount of resolving
done by the atoms in that molecular graph.

Under a ``molecular graph'' we understand a simple graph,
representing the atom skeleton of molecules (chemical compounds).
Thus the vertices of a molecular graph represents the atoms and
edges the atom-atom bonds. Let $G$ be a non-trivial connected graph with vertex set $V(G)$ and edge set $E(G)$. We write
$u\sim v$ if two vertices $u$ and $v$ are adjacent (form an edge) in
$G$ and write $u\not\sim v$ if they are non-adjacent (do not form an
edge). The join of two graphs $G_1$ and $G_2$, denoted by $G_1 +
G_2$, is a graph with vertex set $V(G_1)\cup V(G_2)$ and an edge set
$E(G_1)\cup E(G_2)\cup \{u\thicksim v\ |\ u\in V(G_1)\ \wedge \ v\in
V(G_2)\}$. The $distance$, $d(u,v)$, between two vertices $u$ and
$v$ of $G$ is defined as the length of a shortest $u - v$ path in
$G$, where length is the number of edges in the path. The {\it
diameter} of $G$, denoted by $diam(G)$, is the maximum distance
between any two vertices of $G$. We refer \cite{1} for the general
graph theoretic notations and terminologies not described in this
paper.

A vertex $u$ of $G$ {\it resolves} two distinct vertices $v$ and $w$
of $G$ if $d(v,u)\neq d(w,u)$. A set $R\subseteq V(G)$ is called a
{\em resolving set} for $G$ if every two distinct vertices of $G$
are resolved by some elements of $R$. Such a set $R$ with minimum
cardinality is called a {\it metric basis}, or simply a {\it basis}
of $G$ and that minimum cardinality is called the {\em metric
dimension} of $G$, denoted by $\dim(G)$ \cite{rp2}. Obviously, the
metric dimension of a graph $G$ is a topological index that suggests those minimum number of vertices of $G$ which uniquely determine all
the vertices of $G$ by their shortest distances to the chosen
vertices.

The concept of resolving set was first introduced in the 1970s, by
Slater \cite{slater} and, independently, by Harary and Melter
\cite{rp3}. Slater described the usefulness of this idea into long
range aids to navigation \cite{slater}. Moreover, this concept has
some applications in chemistry for representing chemical compounds
\cite{14,15} and in problems of pattern recognition and image
processing, some of which involve the use of hierarchical data
structures \cite{mel}. Other applications of this concept to
navigation of robots in networks and other areas appear in
\cite{rp2,8}. In recent years, a considerable literature regarding
this notion has developed (see
\cite{javaid,rp2,ftr2,7,ftr1,8,mel,naz,salman1,salman2,salman}).

\section{Resolving Share}
In this section, we define the concept of resolving share and
investigate some basic results which later help in defining and
computing a distance-based topological index. We begin with the
following useful preliminaries: $V_p$ denotes the collection of all
${n\choose 2}$ pairs of the vertices of a graph $G$. For
any vertex $w$ of $G$, let $V_i(w) = \{v\in V(G)-\{w\}\ |\ d(v,w)= i\}$ be the distance neighborhood of $w$  for $\ 1\leq i\leq diam(G)$, and the partition
\[\Pi_w = \{V_i(w)\ ;\ 1\leq i\leq diam(G)\}\] be the distance partition
of the set $V(G)-\{w\}$ with reference of $w$. By $\Pi_w-\{x\}$, we
mean that the vertex $x$ is not lying in any partite set of the
distance partition $\Pi_w$. By $\Pi_x = \Pi_y$, we mean that $V_i(x)
= V_i(y)$ for all $1\leq i\leq diam(G)$.

\begin{Definition}\label{def2.1}
Let $G$ be a connected graph. For any pair $(u,v) \in
V_p$, let $R(u,v)= \{x \in V(G)\ |\ x\ \mbox{resolves}\ u\
\mbox{and}\ v\}$ be the resolving neighborhood of the pair $(u,v)$.
Then for any $w\in V(G)$, the quantity
$$ r_w(u,v) = \left\{
            \begin{array}{ll}
           \frac{1}{|R(u,v)|}&         \,\,\ \mbox{if}\ u\ \mbox{and}\ v\ \mbox{are resolved by w},\hphantom{aaaaaaaaaaaaaaaaaaaaaaaaaaaaaaaa} \\
           \,\,\,\,\ 0&                \,\,\,\ \mbox{otherwise}
            \end{array}
             \right.
$$
is called the {\it resolving share} of $w$ for the pair $(u,v)$.
\end{Definition}

\begin{figure}[h]
        \centerline
        {\includegraphics[width=10cm]{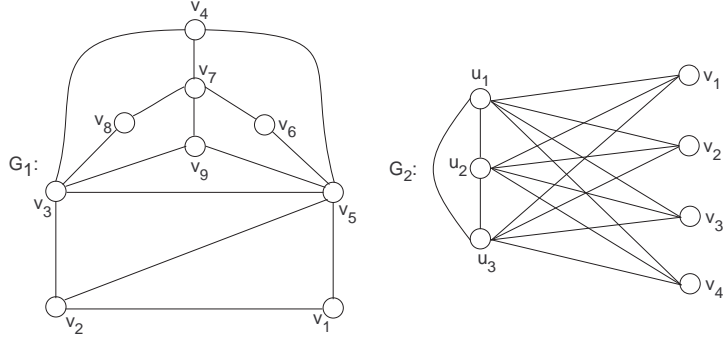}}
        \caption{In the graph $G_1$, the resolving share of the vertex
$v_1$ is zero for the pair $(v_2,v_5)$; is $\frac{1}{4}$ for the
pair $(v_2,v_4)$; is $\frac{1}{5}$ for the pair $(v_2,v_6)$ and is
$\frac{1}{9}$ for the pair $(v_2,v_7)$.}\label{fig1}
\end{figure}

\begin{Remarks}\label{rem2.2}
$(i)$\ The resolving neighborhood of a pair $(u,v)\in V_p$ is the class of all those vertices whose resolving share for the pair $(u,v)$ is same.\\
$(ii)$\ For $w\in V(G)$ and $(u,v)\in V_p$, \[0\leq r_w(u,v)\leq
\frac{1}{2}.\] $(iii)$ $r_u(u,v)\neq 0\neq r_v(u,v)$.
\end{Remarks}

\begin{Lemma}\label{lem2.3}
For a pair $(u,v)\in V_p$ and for a vertex $w\in V(G)-\{u,v\}$,
$r_w(u,v) = 0$ if and only if both $u$ and $v$ belong to the same
partite set of $\Pi_w$.
\end{Lemma}

\begin{proof}
$(\Rightarrow)$\ $r_w(u,v) = 0$ implies that $d(u,w) = d(v,w) = i$
for some $1\leq i\leq diam(G)$. It follows that both $u$ and $v$
belong to the same partite set $V_i(w)\in \Pi_w$.

$(\Leftarrow)$\ If $u,v\in V_i(w)\in \Pi_w$ for some $1\leq i\leq
diam(G)$, then $w$ does not resolve $u$ and $v$, and hence $r_w(u,v)
= 0$.
\end{proof}

\begin{Lemma}\label{lem2.4}
For a pair $(u,v)\in V_p$ and for all $w\in V(G)-\{u,v\}$, $r_w(u,v)
= 0$ if and only if $\Pi_u-\{v\} = \Pi_v-\{u\}$.
\end{Lemma}

\begin{proof}
$(\Rightarrow)$\ Suppose that $r_w(u,v) = 0$ for all $w\in
V(G)-\{u,v\}$. This implies that $d(u,w) = d(v,w)$ for all $w\in
V(G)-\{u,v\}$. Contrarily assume that $\Pi_u-\{v\}\neq \Pi_v-\{u\}$.
It follows that there exists an element $x$ in $V(G)-\{u,v\}$ such
that $x$ lies in a partite set, say $V_i(u)$, of $\Pi_u-\{v\}$ and
$x$ lies in a partite set, say $V_j(v)\ (j\neq i)$, of
$\Pi_v-\{u\}$, and vice-versa. Thus $d(x,u) = i \neq d(x,v)$ or
$d(x,v) = j \neq d(x,u)$, a contradiction. Hence $\Pi_u-\{v\} =
\Pi_v-\{u\}$.

$(\Leftarrow)$\ Suppose that $\Pi_u-\{v\} = \Pi_v-\{u\}$. Assume
contrarily that $r_w(u,v)\neq 0$. It follows that $w\in R(u,v)$ and
hence $d(u,w) \neq d(v,w)$. Thus, there exists a partite set in
$\Pi_u-\{v\}$ which is not equal to any members of $\Pi_v-\{u\}$, a
contradiction. Hence $r_w(u,v)= 0$.
\end{proof}

\begin{Lemma}\label{lem2.5}
For a pair $(u,v)\in V_p$ and for a vertex $w\in V(G)$, $r_w(u,v) =
\frac{1}{2}$ if and only if $w\in \{u,v\}$ and $\Pi_u-\{v\} =
\Pi_v-\{u\}$.
\end{Lemma}

\begin{proof}
$(\Rightarrow)$\ Suppose that $r_w(u,v) = \frac{1}{2}$. Then,
clearly, $w\in \{u,v\}$ because $|R(u,v)| = 2$ and $w$ resolves $u$
and $v$. In fact, $R(u,v) = \{u,v\}$ in this case. It follows that
$d(u,x) = d(v,x)$ for all $x\in V(G)-\{u,v\}$, which concludes that
$\Pi_u-\{v\} = \Pi_v-\{u\}$.

$(\Leftarrow)$\ Suppose that $\Pi_u-\{v\} = \Pi_v-\{u\}$. It follows
that $d(u,z) = d(v,z)$ for all $z\in V(G)-\{u,v\}$, which implies
that the only vertices that resolves the pair $(u,v)$ are the
vertices in the pair. So $R(u,v) = \{u,v\}$, and hence for $w\in
\{u,v\}$, $r_w(u,v) = \frac{1}{2}$.
\end{proof}

\begin{Example}\label{exp2.6}
Consider the graph $G_2$ of Figure \ref{fig1} with vertex
set $V(G_2) = U  = \{u_1,u_2,u_3\}\cup V= \{v_1,v_2,v_3,v_4\}$. Let
$(u,v)\in V_p$ be any pair of vertices of $G_2$. Then note that,
$(i)$\ $\Pi_u-\{v\} = \Pi_v-\{u\}$ for either $u,v\in U$ or $u,v\in
V$; $(ii)$ if $u\in U$ and $v\in V$, then $R(u,v) = \{u\}\cup V$ and
both $u, v$ belong to the same partite set of $\Pi_x$ for all $x\in
U-\{u\}$. Hence, by previous three lemmas, we have
$$ r_w(u,v) = \left\{
            \begin{array}{ll}
           \frac{1}{2}&         \,\,\ \mbox{if}\ w\in \{u,v\},\ \mbox{for either}\ u,v\in U\
           \mbox{or}\ u,v\in V,\\
           \frac{1}{5}&         \,\,\ \mbox{if}\ w\in \{u\}\cup V,\ \mbox{for}\ u\in U\ \mbox{and}\ v\in V,\hphantom{aaaaaaaaaaaaaaaaaaaaaaaaaaaaaaaa}\\
           0&                \,\,\,\ \mbox{otherwise}.
            \end{array}
             \right.
$$
\end{Example}

\begin{Remark}\label{rem2.7}
For each pair $(u,v)\in V_p$, $R(u,v)\cap R \neq \emptyset$ for any
resolving set $R$ for a graph $G$.
\end{Remark}

The following useful result for finding a resolving set for $G$ was
proposed by Chartrand {\em et al.} in 2000.

\begin{Lemma}{\em \cite{rp2}}\label{lem2.8}
Let $R$ be a resolving set for a graph $G$ and $(u,v)\in
V_p$. If $d(u,w) = d(v,w)$ for all $w \in V(G)- \{u,v\}$, then $u$
or $v$ is in $R$.
\end{Lemma}

\begin{Lemma}\label{lem2.9}
Let a pair $(u,v)\in V_p$ and $R$ be any resolving set for $G$.\\
$(1)$\ If $r_w(u,v) = 0$ for all $w\in V(G)-\{u,v\}$, then
$u$ or $v$ is in $R$.\\
$(2)$\ If $r_w(u,v) = \frac{1}{2}$, then $u$ or $v$ is in $R$.
\end{Lemma}

\begin{proof}
$(1)$\ By Lemma \ref{lem2.4}, $\Pi_u-\{v\} = \Pi_v-\{u\}$, which
implies that $d(u,w) = d(v,w)$ for all $w\in V(G)-\{u,v\}$, and
hence Lemma \ref{lem2.8} yields the result.

$(2)$\ By Lemma \ref{lem2.5}, $\Pi_u-\{v\} = \Pi_v-\{u\}$ and $w\in
\{u,v\}$. In fact $R(u,v) = \{u,v\}$ and $d(u,w) = d(v,w)$ for all
$w\in V(G)-\{u,v\}$. Hence, the result follows by Lemma
\ref{lem2.8}.
\end{proof}

Let $|G|$ denotes the order of a graph $G$. The following
assertion is directly follows from the definition of the resolving
share.

\begin{Proposition}\label{prop2.10}
For a pair $(u,v)\in V_p$ and for each vertex $w\in V(G)$, $r_w(u,v)
= \frac{1}{|G|}$ if and only if $R(u,v)= V(G)$.
\end{Proposition}

\begin{Lemma}\label{lem2.11}
For a pair $(u,v)\in V_p$, if $r_w(u,v) = \frac{1}{|G|}$ for each
$w\in V(G)$, then the distance between $u$ and $v$ is odd.
\end{Lemma}

\begin{proof}
If the distance between $u$ and $v$ is even, $i.e.$, $d(u,v) = 2k$
for $k\geq 1$, then there exists a vertex $x$ in $G$ such that
$d(u,x) = k = d(x,v)$, and hence $x\not\in R(u,v)$.
\end{proof}

\begin{Theorem}\label{th2.12}
Let $G$ be a graph with $diam(G) = 2$. Then there are at
most $\left\lfloor \left(\frac{|G|}{2}\right)^2\right\rfloor$ pairs
$(u,v)$ in $V_p$ for which $r_w(u,v) = \frac{1}{|G|}$ for each $w\in
V(G)$.
\end{Theorem}

\begin{proof}
Let $(u,v)\in V_p$ for which $r_w(u,v) = \frac{1}{|G|}$ for each
$w\in V(G)$. Then, $R(u,v) = V(G)$, by Proposition \ref{prop2.10},
and since $diam(G) = 2$ so $d(u,v) = 1$, by Lemma \ref{lem2.11}. It
follows that $u\in V_1(v)\in \Pi_v$ and $v\in V_1(u)\in \Pi_u$.
Moreover, $V_1(u)$ and $V_1(v)$ are disjoint subsets of $V(G)$,
because if there is an element $x$ in $V_1(u)\cap V_1(v)$, then $x$
does not resolve $u$ and $v$ and hence $R(u,v)\neq V(G)$. Also,
$V_1(u)\cup V_1(v) = V(G)$. Otherwise, there exists an element $y$
in $V(G)-\{u,v\}$ such that $u,v\in V_2(y)$, which implies that $u$
and $v$ are not resolved by $y$ yielding $R(u,v)\neq V(G)$. Hence $V_1(u)$ and $V_1(v)$ form a
partition of $V(G)$. Further, for any $a,b \in V_1(u)$ (or $a,b\in
V_1(v)$), $a$ and $b$ have the same distance from $u$ (or $v$) and
hence $R(a,b) \neq V(G)$. It follows that the number of pairs
$(u,v)$ for which $R(u,v) = V(G)$ is bounded above by
\[|V_1(u)||V_1(v)|\leq \left\lfloor
\left(\frac{|G|}{2}\right)^2\right\rfloor.\]
\end{proof}

\begin{Lemma}\label{lem2.13}
If the diameter of a graph $G$ is one, then for every pair
$(u,v)\in V_p$, $r_w(u,v) \in \{0,\frac{1}{2}\}$ for each $w\in
V(G)$.
\end{Lemma}

\begin{proof}
Since $diam(G) = 1$, so $G$ is isomorphic to a complete graph, and for every two vertices $u$ and
$v$ of a complete graph, $\Pi_u-\{v\} = \Pi_v-\{u\}$. Hence, the result
followed by Lemma \ref{lem2.4} if $w\not\in \{u,v\}$, or followed by
Lemma \ref{lem2.5} if $w\in \{u,v\}$.
\end{proof}

\begin{Theorem}\label{th2.14}
Let $G$ be a graph. Then for every pair $(u,v)\in V_p$ and
for each $w\in V(G)$,
$$ r_w(u,v) = \left\{
            \begin{array}{ll}
           \frac{1}{2}&         \,\,\ \mbox{if}\ w\in \{u,v\},\hphantom{aaaaaaaaaaaaaaaaaaaaaaaaaaaaaaaaaaaaaaaa}\\
           0&                \,\,\,\ \mbox{if}\ w\not\in \{u,v\},
            \end{array}
             \right.
$$
if and only if $diam(G) = 1$.
\end{Theorem}

\begin{proof}
Suppose that
$$ r_w(u,v) = \left\{
            \begin{array}{ll}
           \frac{1}{2}&         \,\,\ \mbox{if}\ w\in \{u,v\},\hphantom{aaaaaaaaaaaaaaaaaaaaaaaaaaaaaaaaaaaaaaaaa}\\
           0&                \,\,\,\ \mbox{if}\ w\not\in \{u,v\},
            \end{array}
             \right.
$$ for every pair $(u,v)\in V_p$ and
for each $w\in V(G)$. It follows that $R(u,v) =\{u,v\}$ and $w\in
R(u,v)$ or $w\not\in R(u,v)$. We claim that $diam(G) = 1$. Suppose
that $d(u,v) = k\geq 2$ and $u,u_1,u_2,\ldots,u_{k-1},v$ be a $u-v$
geodesic (shortest path) in $G$ of length $k$. This implies that
$d(u_1,u) = 1$ and $d(u_1,v) = k-1$, and hence $u_1\in R(u,v)$. But
$u_1\not\in \{u,v\}$, a contradiction to the fact that $R(u,v) =
\{u,v\}$. Consequently $d(u,v) = 1$ for every two distinct vertices
$u , v$ of $G$. So $diam(G) = 1$.

The converse part of the theorem followed by Lemma \ref{lem2.13}.
\end{proof}

\section{Resolving Topological Index}
We define the average resolve share of each vertex of a graph and then by
using it we establish a distance based topological index of that graph. Further, we compute the resolving topological index
of certain graphs.

\begin{Definition}\label{def3.1}
Let $G$ be a connected graph. For any vertex $w \in V(G)$,
let $R(w)= \{(u,v) \in V_p\ |\ u\ \mbox{and}\ v\ \mbox{are resolved
by}\ w\}$ be the resolvent neighborhood of the vertex $w$. Then the
quantity \[ar_w(G) = \frac{1}{|R(w)|}\sum \limits_{(u,v) \in R(w)}
r_w(u,v)\hphantom{aaaaaaaaaaaaaaaaaaaaaaaaaaaaaaaaaaaaaaa}\] is the
average of the amount of resolving done by $w$ in $G$, and is
called the {\it average resolving share} of $w$ in $G$.
\end{Definition}

\begin{Remark}\label{rem3.2}
Since for each $x\in V(G)-\{w\}$, $(w,x) \in R(w)$, so $|R(w)|$ and
$ar_w(G)$ will never zero for every $w\in V(G)$.
\end{Remark}

Since $\sum \limits_{(u,v) \in R(w)} r_w(u,v) = |R(w)|a$
$\Leftrightarrow$ $r_w(u,v) = a$ for all $(u,v)\in R(w)$. So, we
have the following straightforward proposition:

\begin{Proposition}\label{prop3.3}
Let $G$ be a graph. For any $w\in V(G)$, $ar_w(G) = a$ if
and only if $r_w(u,v) = a$ for all $(u,v)\in R(w)$.
\end{Proposition}



\begin{Proposition}\label{prop3.4}
Let $G$ be a graph and $w$ be any vertex of $G$. If each
partite set of the distance partition $\Pi_w$ is a singleton set,
then $$ar_w(G) = \frac{2}{|G|(|G|-1)}\sum \limits_{(u,v) \in V_p}
r_w(u,v).\hphantom{aaaaaaaaaaaaaaaaaaaaaaaaaaaaaaaaaaaaaaa}$$
\end{Proposition}

\begin{proof}
If each partite set of the distance partition $\Pi_w$ is a singleton
set, then the vertices of each pair $(u,v)\in V_p$ are resolved by
$w$. It follows that $R(w)  = V_p$ and the proof is complete.
\end{proof}

\begin{Definition}\label{def3.5}
Let $G$ be a graph. Then the positive real number
\[\mathcal{R}(G) = \sum \limits_{w\in V(G)} ar_w(G)\quad \mbox{is called the
resolving topological index of}\,\
G.\hphantom{aaaaaaaaaaaaaaaaaaaaaaaaaaaaaaaaaaaaaaaaaaaaaaaaaa}\]
\end{Definition}

\begin{figure}[h]
        \centerline
        {\includegraphics[width=8cm]{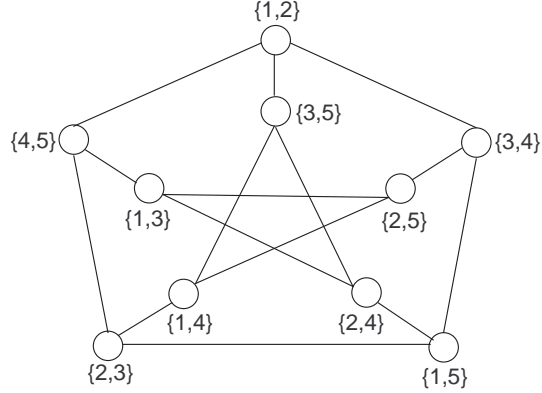}}
        \caption{The Petersen graph}\label{fig2}
\end{figure}

\begin{Theorem}\label{th3.6}
The resolving topological index of the Petersen graph is $\frac{5}{3}$.
\end{Theorem}

\begin{proof}
Let $G$ be the Petersen graph. The vertices of $G$ are the $2$-element subsets of the set
$\{1,2,3,4,5\}$. Let the vertex set of $G$ be $\{S_{ij} = \{i,j\}\
;\ 1\leq i<j\leq 5\}$ and two subsets will be connected by an edge
if their intersection is the empty set (see Figure \ref{fig2}). Let
$(S_{ij},S_{kl})$ be any pair of $V_p$. Then either $S_{ij}\cap
S_{kl} = \emptyset$, or $S_{ij}\cap S_{kl}$ is a singleton set. In
the former case, $i,j\neq k,l$ and $d(S_{ij},S_{kl}) = 1$, and in
the later case, $i = k$, or $i = l$, or $j = k$, or $j = l$ and
$d(S_{ij},S_{kl}) = 2$. Now, we discuss two cases.

\noindent {\bf Case 1.}\ When $S_{ij}\cap S_{kl} = \emptyset$. Then
for each element $A\in V_1(S_{ij})$, $A\cap S_{kl}\neq \emptyset$,
and for each element $B\in V_1(S_{kl})$, $B\cap S_{ij}\neq
\emptyset$. It follows that each $A$ and each $B$ resolves $S_{ij}$
and $S_{kl}$. Further, for each $S_{ab}\in V(G)-(V_1(S_{ij})\cup
V_1(S_{kl}))$, $S_{ab}\cap S_{ij}\neq \emptyset$ as well as
$S_{ab}\cap S_{kl}\neq \emptyset$, and hence $S_{ij},S_{kl}\in
V_2(S_{ab})$. Thus, the resolving neighborhood of the pair
$(S_{ij},S_{kl})$ is $V_1(S_{ij})\cup V_1(S_{kl})$ having six
elements.  Thus, for $X\in V(G)$, by Lemma \ref{lem2.3}, we have
$$ r_X(S_{ij},S_{kl}) = \left\{
            \begin{array}{ll}
           \frac{1}{6}&         \,\,\ \mbox{if}\ X\in V_1(S_{ij})\cup V_1(S_{kl}),\hphantom{aaaaaaaaaaaaaaaaaaaaaaaaaaaaaaaa}\\
           0&                \,\,\ \mbox{otherwise}.
            \end{array}
             \right.
$$
\noindent {\bf Case 2.}\ When $S_{ij}\cap S_{kl} \neq \emptyset$.
Then $(i)$\ $V_1(S_{ij})\cap V_1(S_{kl}) = \{S_{ab}\}$, where
$a,b\neq i,j,k,l$, and hence $S_{ij},S_{kl}$ belong to the same
partite set $V_1(S_{ab})$, $(ii)$\ $|V_2(S_{ij})\cap V_2(S_{kl})| =
3$ and for each $A\in V_2(S_{ij})\cap V_2(S_{kl})$, $S_{ij},S_{kl}$
belong to the same partite set $V_2(A)$. Thus $R(S_{ij},S_{kl}) =
V(G) -[(V_2(S_{ij})\cap V_2(S_{kl}))\cup \{S_{ab}\}]$. Hence, for
$X\in V(G)$, by Lemma \ref{lem2.3}, we have
$$ r_X(S_{ij},S_{kl}) = \left\{
            \begin{array}{ll}
           \frac{1}{6}&         \,\,\ \mbox{if}\ X\not\in V(G) -[(V_2(S_{ij})\cap V_2(S_{kl}))\cup
\{S_{ab}\}],\hphantom{aaaaaaaaaaaaaaaaaaaaaaaaaaaaaaaa}\\
           0&                \,\,\ \mbox{otherwise}.
            \end{array}
             \right.
$$
Note that, for any 2-element subset $X\in V(G)$, $r_X(Y,Z) =
\frac{1}{6}$ for all $(Y,Z)\in R(X)$. Thus, by Proposition
\ref{prop3.3}, $ar_X(G) = \frac{1}{6}$ for all $X\in V(G)$. It
completes the proof.
\end{proof}

\begin{Theorem}\label{th3.7}
Let $G$ be a path on $n\geq 3$ vertices. Then
$$ \mathcal{R}(G)= \left\{
            \begin{array}{ll}
           \sum \limits_{i = 1}^{\frac{n}{2}}\frac{2n^2-3n-4i+4}{n(n-1)^2-2(n-1)(i-1)}&         \,\,\ \mbox{if}\ n\ \mbox{is even},\hphantom{aaaaaaaaaaaaaaaaaaaaaaaaaaaaaaaa}\\
           \frac{2n^2-3n-1}{2n(n-1)^2}+\sum \limits_{i = 1}^{\lfloor\frac{n}{2} \rfloor}\frac{2n^3-3n^2-4n(i-1)+1}{n^2(n-1)^2-2n(n-1)(i-1)}&                \,\,\ \mbox{if}\ n\ \mbox{is odd}.
            \end{array}
             \right.
$$
\end{Theorem}

\begin{proof}
Let $G\ :\ v_1,v_2,\ldots, v_n$ be a path on $n\geq 3$
vertices and let $(v_i,v_j)\in V_p$. Then
$$ R(v_i,v_j) = \left\{
            \begin{array}{ll}
           V(G)&         \,\,\ \mbox{if}\ i+j\ \mbox{is odd},\hphantom{aaaaaaaaaaaaaaaaaaaaaaaaaaaaaaaa}\\
           V(G)-\{v_{\frac{i+j}{2}}\}&                \,\,\ \mbox{if}\ i+j\ \mbox{is even}.
            \end{array}
             \right.
$$
Hence for $w\in V(G)$,
$$ r_w(v_i,v_j) = \left\{
            \begin{array}{ll}
           0&         \,\,\ \mbox{for even}\ i+j,\ \mbox{and}\ w = v_{\frac{i+j}{2}},\hphantom{aaaaaaaaaaaaaaaaaaaaaaaaaaaaaaaa}\\
           \frac{1}{n-1}&         \,\,\ \mbox{for even}\ i+j,\ \mbox{and for all}\ w\in V(G)-\{v_{\frac{i+j}{2}}\},\\
           \frac{1}{n}&                \,\,\,\ \mbox{for odd}\ i+j,\ \mbox{and for all}\ w\in V(G).
            \end{array}
             \right.
$$

Now, since
$$ \Pi_{v} = \left\{
            \begin{array}{ll}
           \{V_j(v)\ ;\ 1\leq j\leq n-i\}&         \,\,\ \mbox{for}\ v\in \{v_i,v_{n-i+1}\},\ 1\leq i\leq \lfloor\frac{n}{2} \rfloor,\hphantom{aaaaaaaaaaaaaaaaaaaaaaaaaaaaaaaa}\\
           \{V_j(v)\ ;\ 1\leq j\leq i-1\}&                \,\,\ \mbox{for}\ v = v_i\ \mbox{and}\ i = \lceil\frac{n}{2} \rceil\, \mbox{for odd}\ n.
            \end{array}
             \right.
$$
So $\Pi_{v_i}$ and $\Pi_{v_{n-i+1}}$ contain $(i-1)$ 2-element
subsets of $V(G)$  for each $1\leq i\leq \lfloor\frac{n}{2} \rfloor$
and $i = \lceil\frac{n}{2} \rceil$. It follows that each $v_i$ and
$v_{n-i+1}$ do not resolve $i-1$ pairs of $V_p$, and hence
\[|R(v_i)| = |R(v_{n-i+1})| = {n\choose 2}-(i-1).\hphantom{aaaaaaaaaaaaaaaaaaaaaaaaaaaaaaaaaaaaaaaaaaaaa}\]
The resolvent neighborhood of each vertex of $G$ consists of two
types of pairs of $V_p$: the pairs $(v_a,v_b)$ for which $a+b$ is
even, we refer them the pairs of type-I; the pairs $(v_a,v_b)$ for
which $a+b$ is odd, we refer them the pairs of type-II. In above, we
have proved that the resolving share of the vertices of $G$ for the
pairs of type-I is $\frac{1}{n-1}$ and for the pairs of type-II is
$\frac{1}{n}$. We discuss the following two cases:

\noindent {\bf Case 1.}\ ($n$ is even)\ Since the sum $a+b$ of two
natural numbers $a,b$ will be even when both $a$ and $b$ are even,
or both $a$ and $b$ are odd, and there are $\frac{n}{2}$ even and
$\frac{n}{2}$ odd natural numbers in the set $\{1,2,\ldots,n\}$ in
this case. So $V_p$ contains $2{\frac{n}{2}\choose 2} =
\frac{1}{4}n(n-2)$ pairs of type-I and ${n\choose 2}-
\frac{1}{4}n(n-2) = \frac{1}{4}n^2$ pairs of type-II. Out of
$\frac{1}{4}n(n-2)$ pairs of type-I, those $\frac{1}{4}n(n-2)-(i-1)$
pairs belong to $R(v_i)$ and $R(v_{n-i+1})$ for which the sum of the
indices of the vertices in the pairs is not equal to $2i$ or
$2(n-i+1)$, respectively. Also, all the $\frac{1}{4}n^2$ pairs of
type-II belong to $R(v_i)$ and $R(v_{n-i+1})$, where $1\leq i\leq
\frac{n}{2}$. It follows that for each $1\leq i\leq \frac{n}{2}$,
\[ar_{v_i}(G) = \frac{2}{n(n-1)-2(i-1)}\left(\frac{n(n-2)}{4(n-1)}-
\frac{i-1}{n-1}+\frac{n}{4}\right) =
ar_{v_{n-i+1}}(G).\hphantom{aaaaaaaaaaaaaaa}\]

\noindent {\bf Case 2.}\ ($n$ is odd)\ Since there are
$\frac{n-1}{2}$ even and $\frac{n+1}{2}$ odd natural numbers in the
set $\{1,2,\ldots,n\}$ in this case. So $V_p$ contains
${\frac{n-1}{2}\choose 2} + {\frac{n+1}{2}\choose 2} =
\frac{1}{4}(n-1)^2$ pairs of type-I and ${n\choose 2}-
\frac{1}{4}(n-1)^2 = \frac{1}{4}(n^2-1)$ pairs of type-II. Out of
$\frac{1}{4}(n-1)^2$ pairs of type-I, those
$\frac{1}{4}(n-1)^2-(i-1)$ pairs belong to $R(v_i)$ and
$R(v_{n-i+1})$ for which the sum of the indices of the vertices in
the pairs is not equal to $2i$ or $2(n-i+1)$, respectively. Also,
all the $\frac{1}{4}(n^2-1)$ pairs of type-II belong to $R(v_i)$ and
$R(v_{n-i+1})$, where $1\leq i\leq \lceil\frac{n}{2} \rceil$. It
follows that for each $1\leq i\leq \lceil\frac{n}{2} \rceil$,
\[ar_{v_i}(G) = \frac{2}{n(n-1)-2(i-1)}\left(\frac{n-1}{4}-
\frac{i-1}{n-1}+\frac{n^2-1}{4n}\right) =
ar_{v_{n-i+1}}(G).\hphantom{aaaaaaaaaaaaaaa}\]

By adding the average resolving shares of all the vertices of $G$ in
the both above cases, one can easily derive the stated resolving
topological index.
\end{proof}

An even path in a graph $G$ is such whose length is even,
and a path is odd if its length is odd.

\begin{Theorem}\label{th3.8}
Let $G$ be a cycle on $n\geq 3$ vertices. Then
$$ \mathcal{R}(G)= \left\{
            \begin{array}{ll}
           \frac{n(n-1)}{n^2-2n+2}&         \,\,\ \mbox{when}\ n\ \mbox{is even},\hphantom{aaaaaaaaaaaaaaaaaaaaaaaaaaaaaaaaaaaaaaaaaaaaaaaaaaaaaaaaaaaaaaaaaaaaa}\\
           \frac{n}{n-1}&                \,\,\ \mbox{when}\ n\ \mbox{is odd}.
            \end{array}
             \right.
$$
\end{Theorem}

\begin{proof}
We consider the following two cases:

\noindent {\bf Case 1.} ($n$ in even) Let $G\ :\ v_1,v_2,\ldots, v_n,v_1$ be a cycle on even $n\geq 4$
vertices and let $D$ be the diameter of $G$. Then for $(v_i,v_j)\in V_p\ (i\neq j)$,

$$ R(v_i,v_j) = \left\{
            \begin{array}{ll}
           V(G)&         \,\,\ \mbox{when}\ i+j\ \mbox{is odd},\hphantom{aaaaaaaaaaaaaaaaaaaaaaaaaaaaaaaa}\\
           V(G)-\{v_{\frac{i+j}{2}}, v_{\frac{i+j}{2}+D}\}&                \,\,\ \mbox{when}\ i+j\ \mbox{is even}.
            \end{array}
             \right.
$$
Hence for $w\in V(G)$,
$$ r_w(v_i,v_j) = \left\{
            \begin{array}{ll}
           0&         \,\,\ \mbox{for even}\ i+j,\ \mbox{and}\ w \in \{v_{\frac{i+j}{2}}, v_{\frac{i+j}{2}+D}\},\hphantom{aaaaaaaaaaaaaaaaaaaaaaaaaaaaaaaa}\\
           \frac{1}{n-2}&         \,\,\ \mbox{for even}\ i+j,\ \mbox{and for all}\ w\in V(G)-\{v_{\frac{i+j}{2}}, v_{\frac{i+j}{2}+D}\},\\
           \frac{1}{n}&                \,\,\,\ \mbox{for odd}\ i+j,\ \mbox{and for all}\ w\in V(G).
            \end{array}
             \right.
$$

Since for each $1\leq i\leq n$, the distance partition $\Pi_{v_i} = \{V_j(v_i)\ ;\ 1\leq j\leq D\}$ contains $(D-1)$ 2-element subsets of $V(G)$, so each $v_i$ do not resolve $D-1$ pairs of $V_p$, and hence
\[|R(v_i)| = {n\choose 2}-(D-1) = \frac{n(n-1)-2(D-1)}{2}.\hphantom{aaaaaaaaaaaaaaaaaaaaaaaaaaaaaaaaaaaaaaaaaaaaaaaaaaaaaaaaaaaaaaaaa}\]

Note that, the resolvent neighborhood of each vertex of $G$ consists of two
types of pairs of $V_p$: the pairs $(v_a,v_b)$ for which $a+b$ is
even, we refer them the pairs of type-I; the pairs $(v_a,v_b)$ for
which $a+b$ is odd, we refer them the pairs of type-II. In above, we
have proved that the resolving share of the vertices of $G$ for the
pairs of type-I is $\frac{1}{n-2}$ and for the pairs of type-II is
$\frac{1}{n}$. Since the sum $a+b$ of two natural numbers $a,b$ will be even when both $a$ and $b$ are even,
or both $a$ and $b$ are odd, and there are $\frac{n}{2}$ even and
$\frac{n}{2}$ odd natural numbers in the set $\{1,2,\ldots,n\}$ in
this case. So $V_p$ contains $2{\frac{n}{2}\choose 2} =
\frac{1}{4}n(n-2)$ pairs of type-I and ${n\choose 2}-
\frac{1}{4}n(n-2) = \frac{1}{4}n^2$ pairs of type-II. Out of
$\frac{1}{4}n(n-2)$ pairs $(v_a,v_b)$ of type-I, those pairs belong to $R(v_i)$ for which $i \neq \frac{a+b}{2}$ and $i \neq \frac{a+b}{2}+D$, and three are $\frac{1}{4}n(n-2)-(D-1)$ such pairs, where $1\leq i\leq
\frac{n}{2}$. Also, all the $\frac{1}{4}n^2$ pairs of
type-II belong to $R(v_i)$, where $1\leq i\leq n$. It follows that for each $1\leq i\leq n$,
\[ar_{v_i}(G) = \frac{2}{n(n-1)-2(D-1)}\left(\frac{n(n-2)}{4(n-2)}-
\frac{D-1}{n-2}+\frac{n^2}{4n}\right) = \frac{n-1}{n^2-2n+2},\hphantom{aaaaaaaaaaaaaaa}\] because $D= \frac{n}{2}$.
It completes the proof.

\noindent {\bf Case 2.} ($n$ is odd) Let $(u,v)$ be any pair of $V_p$ and let $l$ be the length of the
even $u-v$ path in $G$. Then there exists a vertex $x$ in $G$ such
that $u,v\in V_{\frac{l}{2}}(x)\in \Pi_x$. It follows, by Lemma
\ref{lem2.3}, that
$$ r_w(u,v) = \left\{
            \begin{array}{ll}
           0&         \,\,\ \mbox{for}\ w = x,\hphantom{aaaaaaaaaaaaaaaaaaaaaaaaaaaaaaaaaaaaaaaaaaaa}\\
           \frac{1}{n-1}&         \,\,\ \mbox{for all}\ w\in V(G)-\{x\}.
            \end{array}
             \right.
$$

Since the resolving share of each vertex $w$ of $G$ for
every pair in $V_P$ (and hence for every pair in the resolvent
neighborhood of $w$) is $\frac{1}{n-1}$. So Proposition
\ref{prop3.3} yields that $ar_w(G) = \frac{1}{n-1}$ for all $w\in
V(G)$, and it concludes the proof.
\end{proof}

\begin{Theorem}\label{th3.9}
The resolving topological index of a complete graph on at least two vertices is $\frac{|G|}{2}$.
\end{Theorem}

\begin{proof}
Since $diam(G) = 1$, so Theorem \ref{th2.14} implies that for every
pair $(u,v)\in V_p$ and for each $w\in V(G)$,
$$ r_w(u,v) = \left\{
            \begin{array}{ll}
           \frac{1}{2}&         \,\,\ \mbox{if}\ w\in \{u,v\},\hphantom{aaaaaaaaaaaaaaaaaaaaaaaaaaaaaaaaaaaaaaaa}\\
           0&                \,\,\,\ \mbox{if}\ w\not\in \{u,v\}.
            \end{array}
             \right.
$$
Hence $\mathcal{R}(G) = \frac{|G|}{2}$ because $ar_w(G) =
\frac{1}{2}$ for all $w\in V(G)$.
\end{proof}

\begin{Theorem}\label{th3.10}
Let $G$ be a complete $k$-partite
graph $K_{n_1,n_2,\ldots,n_k}$, where each $n_i\geq
2$, $1\leq i\leq k$ and $k\geq 2$ . Then \[\mathcal{R}(G)= \sum \limits_{i=1}^k n_i
\left(\frac{n_i-1}{n_i}+\sum_{\substack{
t=1\\
t\neq i}}^k n_t\right)^{-1}\left(\frac{n_i-1}{2n_i}+\sum_{\substack{
t=1\\
t\neq i}}^k \frac{n_t}{n_i+n_t}\right).\hphantom{aaaaaa}\].
\end{Theorem}

\begin{proof}
Let the partite sets of $G$ are $V_i =
\{v_1^i,v_2^i,\ldots,v_{n_i}^i\}$, where $n_i\geq 2$ and $1\leq
i\leq k$. Let $(x, y)\in V_p$. If $x,y\in V_i$ for any $1\leq i\leq
k$, then $\Pi_x-\{y\} = \Pi_y-\{x\}$. It follows, by Lemmas
\ref{lem2.4} and \ref{lem2.5}, that
$$ r_w(x,y) = \left\{
            \begin{array}{ll}
           \frac{1}{2}&         \,\,\ \mbox{if}\ w\in \{x,y\},\hphantom{aaaaaaaaaaaaaaaaaaaaaaaaaaaaaaaaaaaaaaaaaaaaaaa}\\
           0&                \,\,\,\ \mbox{if}\ w\not\in \{x,y\}.
            \end{array}
             \right.
$$
If $x\in V_i$ and $y\in V_{j\neq i}$, where $1\leq i,j\leq k$. Then
for each $u\in V_i-\{x\}, v\in V_j-\{y\}$ and $z\in V_l$ (for all
$1\leq l\leq k$ and $l\neq i,j$), $x\in V_1(v)\cup V_2(u), y\in
V_1(u)\cup V_2(v)$ and $x,y\in V_1(z)$. It follows that the
resolving neighborhood of $(x,y)$ is $V_i\cup V_j$. Thus, by Lemma
\ref{lem2.3},
$$ r_w(x,y) = \left\{
            \begin{array}{ll}
           \frac{1}{n_i+n_j}&         \,\,\ \mbox{if}\ w\in V_i\cup V_j,\hphantom{aaaaaaaaaaaaaaaaaaaaaaaaaaaaaaaaaaaaaaaaaaaaaaa}\\
           \,\,\,\,\ 0&                \,\,\,\ \mbox{otherwise}.
            \end{array}
             \right.
$$

\noindent For each $1\leq i\leq k$ and for any $v_j^i\in V_i\ (1\leq
j\leq n_i)$, \[R(v_j^i) = \{(v_j^i,v_l^i), (v_a^i,v_b^t)\ ;\ 1\leq
l\neq j\leq n_i,\ 1\leq a\leq n_i,\ 1\leq b\leq n_t,\ 1\leq t\neq
i\leq k\}.\hphantom{aaaaaaaaaaa}\]

\noindent Since
\[r_{v_j^i}(v_j^i,v_l^i) = \frac{1}{2}\quad \mbox{for all}\quad 1\leq l\neq
j\le n_i,\
\mbox{and}\hphantom{aaaaaaaaaaaaaaaaaaaaaaaaaaaaaaaaaaaaaaaaaaaaaaa}\]
\[r_{v_j^i}(v_a^i,v_b^t) = \frac{1}{n_i+n_t}\quad \mbox{for
all}\ 1\leq a\leq n_i,\ 1\leq b\leq n_t,\ 1\leq t\neq i\leq
k.\hphantom{aaaaaaaaaaaaaaaaaaaaaaaaaaaaaaaaaaaaaaaaaaaaaaa}\] So
for each $1\leq i\leq k$ and for each $1\leq j\leq n_i$,
\[ar_{v_j^i}(G) = \frac{1}{|R(v_j^i)|}\left(\sum_{\substack{
l=1\\
l\neq j}}^{n_i} \frac{1}{2} + \sum_{\substack{
t=1\\
t\neq i}}^k \sum_{a=1}^{n_i} \sum_{b=1}^{n_t}
\frac{1}{n_i+n_t}\right)\hphantom{aaaaaaaaaaaaaaaaaaaaaaaaaaaaaaaaaaaaaaaaaaaaaaa}\]
\[= \left[(n_i-1)+n_i\sum_{\substack{
t=1\\
t\neq i}}^k n_t\right]^{-1}\left[\frac{n_i-1}{2}+n_i\sum_{\substack{
t=1\\
t\neq i}}^k \frac{n_t}{n_i+n_t}\right].\hphantom{aaaaaaaaaa}\ \]
\[= \left[n_i\left(1+\sum_{\substack{
t=1\\
t\neq i}}^k
n_t\right)-1\right]^{-1}\left[n_i\left(\frac{1}{2}+\sum_{\substack{
t=1\\
t\neq i}}^k
\frac{n_t}{n_i+n_t}\right)-\frac{1}{2}\right].\hphantom{aaaaaa}\]

Now, by taking the summation of the average resolving shares of the
vertices of all the partite sets of $V(G)$, we get the required
result.
\end{proof}

\begin{Theorem}\label{th3.11}
For each wheel graph $W_n\ (n\geq 6)$, the resolving topological
index is $\frac{(n-3)(n^2+8)}{2(n-2)(4n-13)}$.
\end{Theorem}

\begin{proof}
A wheel graph $W_n$ is the join of a cycle $C_{n-1}$
and the vertex $c$ (called the central vertex of the wheel). First
note that, $\Pi_c = \{V_1(c) = V(C_{n-1})\}$ and for each $v\in
V(C_{n-1})$, $\Pi_v = \{V_1(v), V_2(v)\}$ with $|V_1(v)| = 3$ and
$|V_2(v)| = n-4$ since the diameter of wheel is two. We consider two
cases for any $(x,y)\in V_p$.

\noindent {\bf Case 1.}\ When $d(x,y) = 1$. If $y = c$ and $x\in
V(C_{n-1})$, then each vertex $u$ of $C_{n-1}$ such that $d(u, x) =
2$ belongs to the resolving neighborhood of $(x,y)$. Also, for each
$v\in V_1(x)-\{c\}$, $x$ and $y$ belong to the same partite set of
$\Pi_v$. Thus $R(x,y) = V_2(x)\cup\{x,y\}$, and hence Lemma
\ref{lem2.3} yields that
$$ r_w(x,y) = \left\{
            \begin{array}{ll}
           \,\,\ 0&                \,\,\,\ \mbox{if}\ w\in V_1(x)-\{c\},\hphantom{aaaaaaaaaaaaaaaaaaaaaaaaaaaaaaaaaaaaaaaaaaaaaaa}\\
           \frac{1}{n-2}&         \,\,\ \mbox{otherwise}.
            \end{array}
             \right.
$$
If both $x$ and $y$ belong to $V(C_{n-1})$, then $x\in V_1(y)$ and
$y\in V_1(x)$. In this case, $x,y\in V_1(c)$ and for each $z\in
V_2(x)\cap V_2(y)$, $x,y\in V_2(z)$. Also, for $u\in V_1(x)-\{y,c\}$
and for $v\in V_1(y)-\{x,c\}$, $x\in V_2(v)$ and $y\in V_2(u)$. It
follows that $R(x,y) = \{u,v,x,y\}$, and hence, by Lemma
\ref{lem2.3}, we have
$$ r_w(x,y) = \left\{
            \begin{array}{ll}
            0&                \,\,\,\ \mbox{if}\ w\in \{c\}\cup(V_2(x)\cap V_2(y)),\hphantom{aaaaaaaaaaaaaaaaaaaaaaaaaaaaaaaaaaaaaaaaaaaaaaa}\\
           \frac{1}{4}&         \,\,\ \mbox{otherwise}.
            \end{array}
             \right.
$$
\noindent {\bf Case 2.}\ When $d(x,y) = 2$, then $x,y\in V(C_{n-1})$ and
$x\in V_2(y), y\in V_2(x)$. There are two subcases to discuss.

\noindent {\bf Subcase 2.1.}\ When $|V_1(x)\cap V_1(y)| = 1$. In this
case, for each $u\in V_1(x)-\{c\}$ and for each $v\in V_1(y)-\{c\}$,
$x\in V_2(v)$ and $y\in V_2(u)$. So $R(x,y) = V_2(x)\bigtriangledown
V_2(y)$ and $|R(x,y)| = 6$ (the symbol $X \bigtriangledown Y$
denotes the symmetric difference of two sets $X$ and $Y$). Further,
for each $z\not\in V_2(x)\bigtriangledown V_2(y)$, $x,y$ belong to
the same partite set of $\Pi_z$. It follows, by Lemma \ref{lem2.3},
that
$$ r_w(x,y) = \left\{
            \begin{array}{ll}
            0&                \,\,\,\ \mbox{if}\ w\not\in V_2(x)\bigtriangledown V_2(y),\hphantom{aaaaaaaaaaaaaaaaaaaaaaaaaaaaaaaaaaaaaaaaaaaaaaa}\\
           \frac{1}{6}&         \,\,\ \mbox{otherwise}.
            \end{array}
             \right.
$$
\noindent {\bf Subcase 2.2.}\ When $|V_1(x)\cap V_1(y)| = 2$. In this
case, for $u\in V_1(x)-(V_1(x)\cap V_1(y))$ and for $v\in
V_1(y)-(V_1(x)\cap V_1(y))$, $x\in V_2(v)$ and $y\in V_2(u)$. So
$R(x,y) = V_2(x)\bigtriangledown V_2(y) = \{u,v,x,y\}$. Moreover,
for each $z\not\in \{u,v,x,y\}$, $x,y$ belong to the same partite
set of $\Pi_z$. It concludes, by Lemma \ref{lem2.3}, that
$$ r_w(x,y) = \left\{
            \begin{array}{ll}
            0&                \,\,\,\ \mbox{if}\ w\not\in \{u,v,x,y\},\hphantom{aaaaaaaaaaaaaaaaaaaaaaaaaaaaaaaaaaaaaaaaaaaaaaa}\\
           \frac{1}{4}&         \,\,\ \mbox{otherwise}.
            \end{array}
             \right.
$$
\indent Since $R(c) = \{(c,v)\ ;\ v\in V(C_{n-1})\}$ and each $v\in
V(C_{n-1})$ does not resolve the pairs $(a,b)$ and $(x,y)$ for all
distinct $a,b\in V_1(v)$ and for all distinct $x,y\in V_2(v)$. It
follows that $|R(c)| = n-1$ and for each $v\in V(C_{n-1})$, $|R(v)|
= {n\choose 2} - {3\choose 2}-{n-4\choose 2} = 4n-13$. For $v\in
V(C_{n-1})$, let $V_1(v) =\{c,u,w\}$. Then out of $4n-13$ pairs in
$R(v)$, $(i)$\ $n-3$ pairs are of the form $(c,a),\ a\in
V(W_n)-V_1(v)$, and the resolving share of $v$ for all these pairs
is $\frac{1}{n-2}$; $(ii)$\ $n-2$ pairs are of the form $(v,b),\
b\in V(W_n)-\{c,v\}$, and the resolving share of $v$ for $4$ out of
these $n-2$ pairs is $\frac{1}{4}$ and is $\frac{1}{6}$ for the
remaining $n-6$ pairs; $(iii)$\ $2(n-4)$ pairs are of the form
$(u,d)$ and $(w,d),\ d\in V(W_n)-\{c,u,v,w\}$, and the resolving
share of $v$ for $4$ out of these $2(n-4)$ pairs is $\frac{1}{4}$
and is $\frac{1}{6}$ for the remaining $2(n-6)$ pairs. Also, the
resolving share of the central vertex $c$ for each pair in $R(c)$ is
$\frac{1}{n-2}$. Hence, by Proposition \ref{prop3.3}, $ar_c(W_n) =
\frac{1}{n-2}$ and for each $v\in V(C_{n-1})$,
\[ar_v(W_n) = \frac{1}{4n-13}\left((n-3)\frac{1}{n-2}+4(\frac{1}{4})+(n-6)\frac{1}{6}+4(\frac{1}{4})+2(n-6)\frac{1}{6}\right)\hphantom{aaaaaaaaaaaaaaaaaaa}\]
\[= \frac{n^2-2n-2}{2(n-2)(4n-13)}.\hphantom{aaaaaaaaaaaaaaaaaaaaaaaaaaaaaaaaaaaaa}\]
Hence
\[\mathcal{R}(W_n) = \frac{1}{n-2}+(n-1)\frac{n^2-2n-2}{2(n-2)(4n-13)} =
\frac{(n-3)(n^2+8)}{6(n-2)(4n-13)}.\hphantom{aaaaaaaaaaaaaaaaaaaaaaaaaaaaaaaaaaaaa}\]
\end{proof}

\begin{Theorem}\label{th3.12}
The resolving topological index of a friendship graph
$F_{n}$ is $\frac{2n^3-n^2+4n-4}{4n(3n-2)}$, where $n\geq 2$.
\end{Theorem}

\begin{proof}
A friendship graph $F_n$ is the join $K_1 + G$ and
having $2n+1$ vertices, where $K_1$ is a graph having only
one vertex $c$ (called the central vertex) and $G$ is the graph obtain by taking the union of $n$ copies of the path
$P_2$. For each $u\in V(G)$, $\Pi_u = \{V_1(u), V_2(u)\}$ with
$|V_1(u)| = 2$ and $|V_2(u)| = 2(n-1)$ since the diameter of $F_n$
is two. Let $v\in V(F_n)-\{u\}$. If $v\in V_1(u)$ and $v = c$, then
for the vertex $x\in V_1(u)-\{v\}$, $u$ and $v$ belong to the same
partite set of $\Pi_x$. Also, for each $y\in V_2(u)$, $v\in V_1(y)$
and $u\in V_2(y)$. Hence, together with Lemma \ref{lem2.3} and above
discussion, we have
$$ r_w(u,v) = \left\{
            \begin{array}{ll}
           \ 0&                \,\,\,\ \mbox{if}\ w = x,\hphantom{aaaaaaaaaaaaaaaaaaaaaaaaaaaaaaaaaaaaaaaaaaaaaaa}\\
           \frac{1}{2n}&         \,\,\ \mbox{otherwise}.
            \end{array}
             \right.
$$
If $v\in V_1(u)$ and $v \neq c$, then $u$ and $v$ are the vertices
of the same copy of $P_2$ and $\Pi_u-\{v\} = \Pi_v-\{u\}$. It
follows, by Lemmas \ref{lem2.4} and \ref{lem2.5}, that
$$ r_w(u,v) = \left\{
            \begin{array}{ll}
           0&                \,\,\,\ \mbox{if}\ w\not\in \{u,v\},\hphantom{aaaaaaaaaaaaaaaaaaaaaaaaaaaaaaaaaaaaaaaaa}\\
           \frac{1}{2}&         \,\,\ \mbox{if}\ w\in \{u,v\}.
            \end{array}
             \right.
$$
If $v\in V_2(u)$, then for $a\in V_1(u)-\{c\}$ and for $b\in
V_1(v)-\{c\}$, $u\in V_2(b)$ and $v\in V_2(a)$. Also, for all
$x\not\in (V_1(u)\cup V_1(v))-\{c\}$, $u$ and $v$ belong to the same
partite set of $\Pi_x$. Thus $R(u,v) = \{a,b,u,v\}$, and hence
$$ r_w(u,v) = \left\{
            \begin{array}{ll}
            0&                \,\,\,\ \mbox{if}\ w\in \{c\}\cup(V_1(u)\cup V_1(v)),\hphantom{aaaaaaaaaaaaaaaaaaaaaaaaaaaaaaaaaaaaaaaaaaaaaaa}\\
           \frac{1}{4}&         \,\,\ \mbox{otherwise}.
            \end{array}
             \right.
$$
For any $u\in V(G)$, let $V_1(u) = \{a,c\}$. Then \[R(u) = V_p -
\left(\{(a,c)\}\cup\{(x,v)\ ;\ x,v\in
V_2(u)\}\right)\hphantom{aaaaaaaaaaaaaaaaaaaaaaaaaaaaaaaaaaaaaaaaaaaaaaa}\]
\[= \{(u,w),(a,y), (c,y)\ ;\ w\in V(F_n)-\{u\}, y\in
V(F_n)-\{a,c,u\}\}.\hphantom{aaaaa}\] In $R(u)$, $(i)$\ the number
of pairs of the form $(u,w)$ is $2n$, and the resolving share of $u$
for the pair $(u,a)$ is $\frac{1}{2}$, for the pair $(u,c)$ is
$\frac{1}{2n}$ and for the remanding $2(n-1)$ pairs is $\frac{1}{4}$
; $(ii)$\ the number of pairs of the form $(a,y)$ is $2(n-1)$, and
the resolving share of $u$ for all these pairs is $\frac{1}{4}$;
$(iii)$\ the number of pairs of the form $(c,y)$ is $2(n-1)$, and
the resolving share of $u$ for all these pairs is $\frac{1}{2n}$.
Thus $|R(u)| = 2(3n-2)$, and hence
\[ar_u(F_n) = \frac{1}{2(3n-2)}\left(\frac{1}{2}+\frac{1}{2n}+2(n-1)\frac{1}{4}+2(n-1)\frac{1}{4}+2(n-1)\frac{1}{2n}\right)\hphantom{aaaaaaaaaaaaaaaaaaa}\]
\[= \frac{(n+1)(2n-1)}{4n(3n-2)}.\hphantom{aaaaaaaaaaaaaaaaaaaaaaaaaaaaaaaaaaaaaaa}\]
Since $R(c) = \{(c,v)\ ;\ v\in V(G)\}$ and the resolving share of
$c$ for all these $2n$ pairs is $\frac{1}{2n}$. So, Proposition
\ref{prop3.3} concludes that $ar_c(F_n) = \frac{1}{2n}$. It
concludes that
\[\mathcal{R}(F_n) = \frac{1}{2n}+(n-1)\frac{(n+1)(2n-1)}{4n(3n-2)} =
\frac{2n^3-n^2+4n-4}{4n(3n-2)}.\hphantom{aaaaaaaaaaaaaaaaaaaaaaaaaaaaaaaaaaaaa}\]
\end{proof}



\section{Concluding Remarks}
We investigated the amount of the resolving done by a vertex
$v$ of a graph $G$ for every pair of vertices of $G$,
called the resolving share of $v$, and then we established some
related results. We also quantified the average of the amount of resolving done by $v$ in $G$, and we called it the average resolving share of $v$.
Using average resolving share of each vertex of $G$, we associated a
distance-based topological index with the graph $G$, which
describes the topology of that graph with respect to the
total resolving done by each vertex of that graph, and we called
this topological index, the resolving topological index. Then, by
computing the resolving shares and average resolving shares of all
the vertices, we worked out the resolving topological indices of
certain well-known graphs such as the Petersen graph, paths, cycles,
complete graphs, complete $k$-partite graphs, wheel graphs and
friendship graph. The work done in this paper is a revelation for
the researchers working with resolvability to determine, in
different graphical structures, how they have the topology according
to the resolving done by their vertices. 


\begin{thebibliography}{9}
\bibitem{javaid}
S. Ahamd, M. A. Chaudhry, I. Javaid, M. Salman, On the metric
dimension of the generalized Petersen graphs, {\em Quaestiones
Mathematicae}, {\bf 36}(2013), 421-435.
\bibitem{33}
A. T. Balaban, {\em Chem. Phys.}, {\bf 89}(1982), 399.
\bibitem{rp2}
G. Chartrand, L. Eroh, M. A. Johnson, O. R. Oellermann,
Resolvability in graphs and the metric dimension of a graph, {\em
Disc. Appl. Math.}, {\bf 105}(2000), 99-113.
\bibitem{1}
G. Chartrand, L. Lesniak, Graphs and Digraphs, $3$rd ed., Chapman
and Hall, London, 1996.
\bibitem{ftr2}
M. A. Chaudhry, I. Javaid, M. Salman, Fault-Tolerant metric and
partition dimension of graphs, {\em Util. Math.}, {\bf 83}(2010),
187-199.
\bibitem{3}
M. V. Diudea, I. Gutman, {\em Croat. Chem. Acta}, {\bf 71}(1998),
21-51.
\bibitem{4}
I. Gutman, A formula for the Wiener number of trees and its
extension to graphs containing cycles, {\em Graph theory notes},
{\bf 27}(1994), 9-15.
\bibitem{6}
I. Gutman, D. Vuki\v{c}evi\v{c}, J. \v{Z}erovnik, A class of
modified Wiener indices, {\em Croat. Chem. Acta}, {\bf 77}(2004),
103-109.
\bibitem{rp3}
F. Harary, R. A. Melter, On the metric dimension of a graph, {\em
Ars Combin.}, {\bf 2}(1976), 191-195.
\bibitem{7}
C. Hernando, M. Mora, I. M. Pelayoe, C. Seara, D. R. Wood, Extremal
graph theory for metric dimension and diameter, {\em Elect. J. of
Combin.}, {\bf 17}(2010), no. R30.
\bibitem{ftr1}
I. Javaid, M. Salman, M. A. Chaudhry, S. Shokat, Fault-Tolerance in
resolvability, {\em Util. Math.}, {\bf 80}(2009), 263-275.
\bibitem{14}
M. A. Johnson, Structure-activity maps for visualizing the graph
variables arising in drug design, {\em J. Biopharm. Statist.}, {\bf
3}(1993), 203-236.
\bibitem{15}
M. A. Johnson, Browsable structure-activity datasets, Advances in
molecular similarity (R. Carb\'{o}-Dorca and P. Mezey, eds.), {\em
JAI Press Connecticut}, (1998), 153-170.
\bibitem{8}
S. Khuller, B. Raghavachari, A. Rosenfeld, Landmarks in graphs {\em
Disc. Appl. Math.}, {\bf 70}(1996), 217-229.
\bibitem{10}
I. Lukovits, in: M. V. Diudea (Ed.), QSPR/QSAR Studies by molecular
descriptors, {\em Nova, Huntigton}, (2001), 31-38.
\bibitem{mel}
R. A. Melter, I. Tomescu, Metric bases in digital geometry {\em
Computer Vision Graphics and Image Processing}, {\bf 25}(1984),
113-121.
\bibitem{naz}
S. Naz, M. Salman, U. Ali, I. Javaid, S. A. Bokhary, On the constant
metric dimension of generalized Petersen graps $P(n,4)$, {\em Acta.
Math. Sinica, Englich Series}, in press.
\bibitem{13}
S. Nikoli\'{c}, N. Trinajsti\'{c}, M. Randi\'{c}, Wiener index
revisited, {\em Chem. Phys. Lett.}, {\bf 333}(2001), 319-321.
\bibitem{12}
S. Nikoli\'{c}, N. Trinajsti\'{c}, Z. Mihali\'{c}, The Wiener index:
Development and applications, {\em Croat. Chem. Acta}, {\bf
68}(1995), 105-129.
\bibitem{29}
D. Plasvi\'{c}, S. Nikoli\'{c}, N. Trinajsti\'{c}, Z. Mihali\'{c},
{\em J. Math. Chem.}, {\bf 12}(1993), 235.
\bibitem{18}
M. Randi\'{c}, On generalization of Wiener index for cyclic
structures, {\em Acta Chemica Slovenica}, {\bf 29}(2000), 483-496.
\bibitem{17}
M. Randi\'{c}, {\em J. Chem. Educ.}, {\bf 69}(1992), 713-718.
\bibitem{32}
M. Randi\'{c}, {\em J. Am. Chem.}, {\bf 97}(1975), 6609.
\bibitem{salman1}
M. Salman, M. A. Chaudhry, I. Javaid, On the locatic number of
graphs, {\em Int. J. Comp. Math.}, {\bf 90}(5)(2013), 912-920.
\bibitem{salman2}
M. Salman, I. Javaid, M. A. Chaudhry, 2-size resolvability in
graphs, {\em Appl. Math. Inf. Sci.}, {\bf 6}(2)(2012), 371-376.
\bibitem{salman}
M. Salman, I. Javaid, M. A. Chaudhry, Resolvability in circulant
graphs, {\em Acta. Math. Sinica, Englich Series}, {\bf 28}(9)(2012),
1851-1864.
\bibitem{20}
H. P. Schultz, Topological organic chemistry 1. Graph theory and
topological indices of alkanes, {\em J. Chem. Inf. Comput. Sci.},
{\bf 29}(1989), 227-228.
\bibitem{slater}
P. J. Slater, Leaves of trees, {\em Congr. Numer.}, {\bf 14}(1975),
549-559.
\bibitem{23}
N. Trinajsti\'{c}, Chemical graph theory, {\em CRC Press, Boca
Raton}, 1983; 2nd revised ed. 1992.
\bibitem{25}
H. Wiener, Correlation of heats of isomerization, and differences in
heats of vaporization of isomers, among the paraffin hydrocarbons,
{\em J. Am. Chem. Soc.}, {\bf 69}(1947), 17-20.
\end{thebibliography}
\end{document}